\documentclass[reqno]{article}
\usepackage{amsmath, amsfonts, amsthm, amssymb, eucal, indentfirst}

\newtheorem{theorem}{Theorem}
\newtheorem{lemma}[theorem]{Lemma}
\newtheorem{corollary}[theorem]{Corollary}

\begin{document}

\title{Extremal results on average subtree density of series-reduced trees}
\author{John Haslegrave
\thanks{Trinity College, Cambridge CB2 1TQ}}

\maketitle

\begin{abstract}
Vince and Wang \cite{VW10} showed that the average subtree density of a series-reduced tree is between
$\frac{1}{2}$ and $\frac{3}{4}$, answering a conjecture of Jamison \cite{Jam83}. They ask under what 
conditions a sequence of such trees may have average subtree density tending to either bound; we answer 
these questions by giving simple necessary and sufficient conditions in each case.
\end{abstract}

\section{Introduction}
In this paper we consider the average order over all subtrees of a fixed tree $T$. We call this quantity 
$\mu(T)$, and the average subtree density $D(T)=\frac{\mu(T)}{n(T)}$ where $n(T)$ is the number of 
vertices of $T$. Equivalently, we may think of $\mu(T)$ as the sum over vertices of $T$ of the proportion 
of subtrees which contain that vertex, and $D(T)$ as the average over vertices of $T$ of the proportion 
of subtrees which contain that vertex; we shall frequently use these alternative definitions.

These invariants were introduced by Jamison \cite{Jam83}, who studied the extremal problem, showing 
that the tree of order $n$ which minimises the average order of a subtree is the path $P_n$, for 
which $\mu(P_n)=\frac{n+2}{3}$, but that there exist trees with $D(T)$ arbitrarily close to 1. Meir and 
Moon \cite{MM83} gave asymptotic results on the average value of $D(T)$ over all trees of order $n$.

We shall focus our attention on trees whose internal vertices all have degree at least three; we write 
$\mathcal{T}_3$ for the set of trees with at least one internal vertex and all internal vertices having 
degree at least three (so we exclude the one- and two-vertex trees which would otherwise vacuously satisfy 
the condition). Such trees are sometimes referred to as series-reduced trees or homeomorphically irreducible
trees \cite{BLL}, and have been studied in other contexts (e.g. \cite{HP59}). (As usual, we refer to 
vertices of degree at least 2 as ``internal vertices'' of a tree, and other vertices as ``leaves''.)

For trees $T \in \mathcal{T}_3$, Jamison \cite{Jam83} conjectured that $D(T)\ge\frac{1}{2}$; this, along 
with the upper bound $D(T)<\frac{3}{4}$, was proved by Vince and Wang \cite{VW10}, who also asked 
questions about under what conditions a sequence of distinct trees in $\mathcal{T}_3$ can have density 
tending to either limit. We shall give a simpler proof of their upper bound, as well as answering their
questions by giving exact necessary and sufficient conditions for a sequence of trees in $\mathcal{T}_3$
to have density tending to $\frac{1}{2}$ or $\frac{3}{4}$.

\section{The upper bound}
For a given tree $T$, the subgraphs of $T$ which are also trees must all be induced subgraphs. Write 
$\mathcal{S}(T)$ for the set of subsets of $V(T)$ which induce a subgraph which is a tree. We require 
each subtree to have at least one vertex, so that $\varnothing\notin \mathcal{S}(T)$ but 
$\{v\}\in\mathcal{S}(T)$ for every $v\in V$.

First we make a simple observation which immediately gives us the upper bound $D(T)<\frac{3}{4}$ for 
$T \in \mathcal{T}_3$.

\begin{lemma}\label{stleaf}
If $v$ is a leaf of a tree $T$ with at least four vertices, less than half of the subtrees of $T$ contain $v$.
\end{lemma}
\begin{proof}
Write $u$ for the neighbour of $v$. For every $S\in\mathcal{S}(T)$ with $v\in S$ and $S\neq \{v\}$ 
write $S'=S-v$; then $S'\in \mathcal{S}(T)$, and $u\in S'$. Conversely, if $u\in R$, $v\notin R$ 
and $R\in \mathcal{S}(T)$ then $R=S'$ for some $S\in \mathcal{S}(T)$ containing $v$.
Thus there is a one-to-one correspondence between subtrees containing $v$ 
(other than singleton $v$) and subtrees containing $u$ but not $v$. 
Since there are at least two subtrees which contain neither $u$ nor $v$ (namely any other single vertex), 
there are more subtrees which do not contain $v$ than subtrees which do.
\end{proof}
Any tree (other than the single-vertex tree) has at least two leaves; if $T \in \mathcal{T}_3$ then we 
can say more, giving us our desired result.
\begin{lemma}
If $T \in \mathcal{T}_3$ has $n$ vertices then $\mu(T)<\frac{3n-2}{4}$.
\end{lemma}
\begin{proof}
Since $T$ has $n-1$ edges, $\sum_v d(v)=2n-2$. If $T$ has $l$ leaves then each of the
$n-l$ internal vertices has degree at least 3, so $2n-2\ge l + 3(n-l)$, i.e. $l \ge \frac{n+2}{2}$. 
Recall that $\mu(T)$ is the sum over vertices of the proportion of subtrees containing that vertex. 
By definition, $T$ has at least four vertices, so we may apply Lemma~\ref{stleaf} to get
$\mu(T)<\frac{l}{2}+(n-l)\le \frac{3n-2}{4}$.
\end{proof}

We shall frequently recall the observation above that if a tree in $\mathcal{T}_3$ has $n$ vertices and 
$l$ leaves then $l \ge \frac{n+2}{2}$.

Vince and Wang \cite{VW10} ask under what conditions a sequence of distinct trees in $\mathcal{T}_3$ can 
have densities tending to $\frac{1}{2}$. They consider the conditions of bounded diameter and of unbounded 
degree, observing that neither is sufficient. In fact the bounded diameter condition (which would imply 
unbounded degree) is not necessary either. Using Lemma \ref{stleaf} we can see that a sufficient condition
is that the proportion of leaves tends to 1, since if $T$ has $n$ vertices and at least $(1-\varepsilon)n$ 
leaves then $\mu(T)<\varepsilon n+\frac{1}{2}(1-\varepsilon)n$ and so $D(T)<\frac{1}{2}(1+\varepsilon)$. 
We can certainly construct sequences of trees for which the diameter is unbounded but the proportion of leaves 
tends to 1, such as the tree formed by connecting $n$ stars of order $n$ by a path. We might hope that 
this condition on the proportion of leaves is also necessary; in the next section we shall show that this is so.

Vince and Wang \cite{VW10} also give a sequence of trees in $\mathcal{T}_3$ with density tending to 
$\frac{3}{4}$ and go on to ask about necessary and sufficient conditions for a sequence of 
distinct trees in $\mathcal{T}_3$ to have density tending to $\frac{3}{4}$. They suggest that suitable 
conditions might be that both the proportion of leaves tends to $\frac{1}{2}$ (the minimum possible limit for 
trees in $\mathcal{T}_3$) and the ratio of the diameter to the number of vertices tends to $\frac{1}{2}$ (the 
maximum possible limit for trees in $\mathcal{T}_3$). From Lemma~\ref{stleaf} it is clear that the first 
condition is necessary. The second is not, however, as we shall see.

\section{Twigs and improved upper bounds}
Define a vertex $v$ of a tree to be a \textit{twig} if $d(v)\ge 2$ but at least $d(v)-1$ of its neighbours 
are leaves. An equivalent definition is that the twigs of $T$ are the leaves of the tree $T'$ formed by 
deleting all the leaves of $T$. In this section we aim to show that trees for which $D(T)$ is close to 
$\frac{3}{4}$ must have few twigs.

We have already defined $\mathcal{S}(T)$ for a given tree $T$. Let $\mathcal{S}'(T)$ be obtained from 
$\mathcal{S}(T)$ by adding $\varnothing$ but removing $\{v\}$ for every leaf $v$ (other singleton sets 
remain). $\mu(T)$ is the average order of a subset of $V(T)$ in $\mathcal{S}(T)$. 
Here we shall find it more convenient to consider $\mu'(T)$, the average order of a subset in 
$\mathcal{S}'(T)$, and so we wish to establish an inequality between the two.
\begin{lemma}
For any tree $T$ with at least four vertices, $\mu(T) \le \mu'(T)$, with equality only for the path on 
four vertices, $P_4$.
\end{lemma}
\begin{proof}
Suppose that $T$ has $n$ vertices and $l$ leaves. Consider the subsets of $V(T)$ in 
$\mathcal{S}(T) \cap \mathcal{S}'(T)$, that is to say the nonempty subsets which induce subtrees 
which are not single leaves of $T$. Let the number of such subsets be $a$ and the total of their orders be $A$, 
then $\mu(T)=\frac{A+l}{a+l}$, $\mu'(T)=\frac{A}{a+1}$. Firstly, we claim that $\mu'(T)\ge 2$: 
since the $a$ subtrees in $\mathcal{S}(T) \cap \mathcal{S}'(T)$ comprise $n-l$ of order 1, one of order $n$, 
$l$ of order $n-1$ obtained by removing a leaf from $T$, and $a-n-1$ others, each of order at least 2,
\begin{eqnarray*}
A&\ge& 2(a-n-1)+(n-l)+n+l(n-1) \\
&=& 2a +(n-2)l - 2 \\
&\ge& 2a + 2.
\end{eqnarray*}
The last line follows since $n\ge 4$ and $l\ge 2$; equality therefore occurs only when $n=4$ and $l=2$, i.e. $T=P_4$.

Now, since $l\ge 2$, $\frac{l}{l-1}\le 2\le \mu'(T)$ (with equality only if $T=P_4$). So
\begin{eqnarray*}
\mu(T)=\frac{A+l}{a+l}&=&\frac{a+1}{a+l}\frac{A}{a+1}+\frac{l-1}{l+a}\frac{l}{l-1} \\
&\le& \frac{a+1}{a+l}\mu'(T)+\frac{l-1}{l+a}\mu'(T) \\
&=&\mu'(T),
\end{eqnarray*}
as required, again with equality only if $T=P_4$.
\end{proof}
\begin{lemma}\label{twigub}
If $T\in \mathcal{T}_3$ has $n\ge 4$ vertices and $t$ twigs then $\mu(T) < \frac{3n}{4}-\frac{2t}{5}$.
\end{lemma}
\begin{proof}
We shall instead show that $\mu'(T)<\frac{3n}{4}-\frac{2t}{5}$; by the previous lemma, this is sufficient.
$\mu'(T)$ is the average number of vertices belonging to a subset in $\mathcal{S}'(T)$; equivalently, it is 
the sum over all vertices of the proportion of subsets in $\mathcal{S}'(T)$ which contain that vertex.

If $w$ is any leaf not adjacent to a twig then removing $w$ from any set in $\mathcal{S}'(T)$ containing 
it gives another set in $\mathcal{S}'(T)$, and all sets obtained in this way are distinct, so $w$ is 
in at most half of the sets in $\mathcal{S}'(T)$. Certainly any vertex which is neither a twig nor a 
leaf is in at most all sets in $\mathcal{S}'(T)$.

Suppose $T$ has $l_1$ leaves which are adjacent to twigs and $l_2$ leaves not adjacent to twigs. 
Consider the tree $T'$ obtained from $T$ by removing the $l_1$ leaves adjacent to twigs of $T$. 
Each vertex which was a twig in $T$ is now a leaf in $T'$ and any other internal vertex of $T$ has the same 
degree in $T'$, so either $T'\in\mathcal{T}_3$ or $T'$ is the one-vertex tree or $T'$ is the two-vertex tree.
Recall that any tree in $\mathcal{T}_3$ with $k$ vertices has at least $\frac{k+2}{2}$ leaves. 
$T'$ has $n-l_1$ vertices and $l_2+t$ leaves, so if $T'\in\mathcal{T}_3$ then $l_2+t\ge\frac{n-l_1+2}{2}$, 
i.e. $l_2\ge\frac{n-l_1-2t+2}{2}$. If $T'$ is the one-vertex tree then $l_1=n-1$, $t=1$ and $l_2=0$, 
so $l_2=\frac{n-l_1-2t+1}{2}$. If $T'$ is the two-vertex tree then $l_1=n-2$, $t=2$ and $l_2=0$, 
so $l_2=\frac{n-l_1-2t+2}{2}$. In each case $l_2\ge\frac{n-l_1-2t+1}{2}$.

For any twig, $v$, with $a$ adjacent leaves, we define an equivalence relation on $\mathcal{S}'(T)$ by 
saying two sets are equivalent if they differ only on $v$ and adjacent leaves. Since any set in 
$\mathcal{S}'(T)$ which contains a leaf adjacent to $v$ must also contain $v$, each equivalence class 
consists either of a single set which does not contain $v$, or of $2^a+1$ sets, one not containing $v$ and 
the remainder containing $v$ and any subset of adjacent leaves. Thus the proportion of sets in 
$\mathcal{S}'(T)$ which contain $v$ is at most $\frac{2^a}{2^a+1}$, and the proportion which contain
any given leaf adjacent to $v$ is at most $\frac{2^{a-1}}{2^a+1}$. Note that, since $d(v)\ge 3$, $a\ge 2$.

Let the twigs of $T$ be $v_1, v_2, \ldots, v_t$, with $a_1, a_2, \ldots, a_t$ leaves respectively, so that 
$l_1=\sum a_i\ge 2t$. The sum, as $u$ ranges over all twigs and leaves adjacent to twigs, of the proportion of sets in $\mathcal{S}'(T)$ which contain $u$ is therefore
\begin{eqnarray*}
\sum_{i=1}^t\frac{2^{a_i}+a_i2^{a_i-1}}{2^{a_i}+1}&\le&\sum_{i=1}^t\frac{28a_i+16}{45} \\
&=&\frac{28l_1+16t}{45},
\end{eqnarray*}
where the first inequality is the result of Lemma \ref{stpoly}, following. We shall now combine all 
these bounds to bound $\mu'_T$.
\begin{eqnarray*}
\mu'(T)&=&\sum_{v\in V(T)}\frac{|\{S\in \mathcal{S}'(T):v\in S\}|}{|\mathcal{S}'(T)|} \\
&\le&(n-t-l_1-l_2)+\frac{l_2}{2}+\frac{28l_1+16t}{45} \\
&=&n-\frac{29t}{45}-\frac{17l_1}{45}-\frac{l_2}{2} \\
&\le&n-\frac{29t}{45}-\frac{17l_1}{45}-\frac{n-l_1-2t+1}{4} \\
&=&\frac{3n}{4}-\frac{13t}{90}-\frac{23l_1}{180}-\frac{1}{4} \\
&\le&\frac{3n}{4}-\frac{2t}{5}-\frac{1}{4},
\end{eqnarray*}
where the final inequality follows from the observation that $l_1\ge 2t$.
\end{proof}
\begin{lemma}\label{stpoly}If $a\ge 2$ is an integer then $\frac{2^a+a2^{a-1}}{2^a+1}\le\frac{28a+16}{45}$.
\end{lemma}
\begin{proof}
Since
\begin{equation*}
\frac{28a+16}{45}-\frac{2^a+a2^{a-1}}{2^a+1}=\frac{11a2^a+56a+32-58\times 2^a}{90\left(2^a+1\right)},
\end{equation*}
it is sufficient to show that
\begin{equation*}
11a2^a+56a+32\ge 58\times 2^a.
\end{equation*}
If $a\ge 6$ then $11a2^a\ge 58\times 2^a$, so it only remains to check for $a=2,3,4,5$. For $a=2$ and $a=3$ 
the LHS and RHS are equal; for $a=4$ and $a=5$ the LHS is greater.
\end{proof}
If $T$ has $n$ vertices and at least $\varepsilon n$ twigs, then, 
$D(T)=\frac{1}{n}\mu(T)<\frac{3}{4}-\frac{2}{5}\varepsilon$, and so for a sequence of trees to have 
average subtree density tending to $\frac{3}{4}$ it is necessary for the proportion of twigs to tend to 0.
In the previous section we showed that it is also necessary for the proportion of leaves to tend to 
$\frac{1}{2}$. We might hope that it is necessary and sufficient for both the proportion of leaves to tend 
to $\frac{1}{2}$ and the proportion of twigs to tend to 0; we shall later show that this is the case.

\section{Rooted approximations}
In this section we will consider trees with a designated vertex, the \textit{root}. We shall compare 
the average order of a subtree to the average order of a subtree containing the root. For trees in 
$\mathcal{T}_3$, we shall show that for a suitably chosen root these two quantities differ by at most a 
constant.

Let $\mathcal{T}^*_3$ be the set of rooted trees such that the root has degree at least two and every 
other internal vertex has degree at least three, together with the single-vertex rooted tree. The 
motivation for this definition is that any tree in $\mathcal{T}_3$, when rooted at any internal vertex, 
is in $\mathcal{T}^*_3$, and that the definition of $\mathcal{T}^*_3$ permits induction in the following 
manner. For any $T\in \mathcal{T}^*_3$ other than the single-vertex tree, we may delete the root to leave 
two or more components. We may consider each component as a new tree; it must contain exactly one neighbour 
of the deleted root and, if rooted at that vertex, is in $\mathcal{T}^*_3$. In some cases we shall prove 
results for all rooted trees using a similar inductive process.

For a tree $T$ with a vertex $v$ write $\alpha(T,v)$ for the number of subtrees containing $v$ and 
$\bar{\alpha}(T,v)$ for the number not containing $v$; likewise, write $\lambda(T,v)$ and $\bar{\lambda}(T,v)$ 
for the average order of subtrees containing and not containing $v$ respectively. 
(Recall that, by definition, all subtrees are non-empty.) Similarly define $\alpha(T,e)$, etc., when $e$ is 
an edge. We start by observing that we may calculate $\lambda(T,v)$ in terms of the parameters of the 
components of $T-v$. This observation also appears in the papers of Jamison \cite{Jam83} and Vince and 
Wang \cite{VW10}.

\begin{lemma}\label{stsetup}
Let $T$ be a rooted tree with root $v$. If $v$ has $d$ neighbours
$v_1,v_2,\ldots v_d$, let $T_i$ for $1\le i \le d$ be the component of $T-v$
containing $v_i$. Then
\begin{equation*}
\lambda(T,v) =
1+\sum_{i=1}^d\frac{\lambda(T_i,v_i)\alpha(T_i,v_i)}{\alpha(T_i,v_i)+1}.
\end{equation*}
\end{lemma}
\begin{proof}
If $S$ is chosen uniformly from all subtrees of $T$ which contain $v$, the
average number of vertices in $S$ is the sum of the average number of vertices
in each $S\cap T_i$, plus one (for $v$ itself, which is not in any $T_i$). For
each $i$, consider the equivalence relation where two subtrees $S_1, S_2$ are
equivalent if they differ only on vertices of $T_i$;
since each subtree contains $v$, each equivalence class contains
$\alpha(T_i,v_i)+1$ subtrees, one with empty intersection with $T_i$, and the
others having intersections corresponding to the subtrees of $T_i$ containing
$v_i$. The sum of the orders of these intersections is therefore
$\alpha(T_i,v_i)\lambda(T_i,v_i)$, so the average intersection over each
equivalence class, and hence the overall average intersection with $T_i$, is
$\frac{\lambda(T_i,v_i)\alpha(T_i,v_i)}{\alpha(T_i,v_i)+1}$. Summing over
$i$ gives the desired result.
\end{proof}

We shall need the following bound, which appears in the paper of Vince and Wang \cite{VW10}; 
we include their proof for completeness.
\begin{lemma}[\cite{VW10}]\label{stroot}
If $T\in\mathcal{T}^*_3$ with root $v$ then $\alpha(T,v)\ge \bar{\alpha}(T,v)$.
\end{lemma}
\begin{proof}
We use induction on the order of $T$. For the one-vertex tree the result is trivial; otherwise $v$ has at 
least two neighbours, $v_1, v_2, \ldots, v_d$, and $T-v$ has corresponding components $T_1, T_2, \ldots, T_d$. Then
\begin{equation*}
\alpha(T,v)=\prod_{i=1}^d(\alpha(T_i,v_i)+1),
\end{equation*}
since a subtree containing $v$ consists of $v$ together with either a subtree containing $v_i$ or 
no vertices from each $T_i$. Also
\begin{equation*}
\bar{\alpha}(T,v)=\sum_{i=1}^d(\alpha(T_i,v_i)+\bar{\alpha}(T_i,v_i)),
\end{equation*}
since each subtree not containing $v$ is a subtree of one of the $T_i$.

We claim that, for any positive integers $a_1, a_2, \ldots, a_d$, with $d\ge 2$, \\
$\prod_{i=1}^d(a_i+1)\ge 2\sum_{i=1}^da_i$. We prove this by induction on $d$. 
For $d=2$, $(a_1+1)(a_2+1)-2(a_1+a_2)=(a_1-1)(a_2-1)\ge 0$; for $d\ge 3$, 
\begin{eqnarray*}
\prod_{i=1}^d(a_i+1)&=&(a_d+1)\prod_{i=1}^{d-1}(a_i+1) \\
&\ge& 2(a_d+1)\sum_{i=1}^{d-1}a_i \\
&\ge& 2a_d+2\sum_{i=1}^{d-1}a_i,
\end{eqnarray*}
as required.

Hence, using the induction hypothesis for $T_i$,
\begin{eqnarray*}
\alpha(T,v)&=&\prod_{i=1}^d(\alpha(T_i,v_i)+1) \\
&\ge& 2\sum_{i=1}^d\alpha(T_i,v_i) \\
&\ge& \sum_{i=1}^d(\alpha(T_i,v_i)+\bar{\alpha}(T_i,v_i)) \\
&=& \bar{\alpha}(T,v),
\end{eqnarray*}
as required.
\end{proof}

We shall also need a lower bound on $\alpha(T,v)$.

\begin{lemma}\label{sttreecount}Let $T$ be any rooted tree with root $v$, $n$ vertices and $l$ leaves, 
not counting the root as a leaf. Then $\alpha(T,v)\ge n-l-1+2^l$. 
\end{lemma}
\begin{proof}
For each vertex $w$, consider the subtree $S_w$ consisting of a path from $v$
to $w$ (the single-vertex path if $v=w$). Each such subtree is distinct (since
if $v\ne w$ then $w$ is the unique leaf of $S_w$ other than $v$) and contains
$v$, so $\alpha(T,v)\ge n$. In addition, if $l>1$, consider the subtrees consisting
of all vertices other than the leaves (including the root, which we do not count as a leaf)
together with any set of at least two leaves. All such subtrees are distinct and
are not paths ending at the root, so not equal to $S_w$ for any $w$. There are $2^l-l-1$ such subtrees,
so $\alpha(T,v)\ge n-l-1+2^l\ge n$.
\end{proof}
This bound is best possible for any $l<n$, as seen by considering the tree consisting of a vertex adjacent 
to $l$ leaves and connected to the root by a path of length $n-l-1$ (by which we mean that this vertex is 
itself the root when $n-l-1=0$).

Since the root is not counted among the leaves, even if it has degree 1, $n-l-1\ge 0$ and so we shall 
sometimes use the weaker bound $\alpha(T,v)\ge 2^l$.

Now we shall use these bounds to show that, for any $T\in \mathcal{T}_3$ with at least 30 vertices, we 
may approximate $\mu(T)$ by $\lambda(T,v)$ for some suitable choice of $v$.
\begin{lemma}\label{stapprox1}
If $T\in \mathcal{T}_3$ is a tree with $n\ge 30$ vertices then either there is an edge $e$ for which 
$2\alpha(T,e)\ge n\bar{\alpha}(T,e)$ or there is an internal vertex $v$ for which 
$2\alpha(T,v)\ge n\bar{\alpha}(T,v)$.
\end{lemma}
\begin{proof}
Pick any edge $e=\{v_1,v_2\}$. Let $T_1,T_2$ be the components of $T-e$ containing $v_1,v_2$ respectively. 
Then each subtree not containing $e$ is a subtree of one component of $T-e$, and each subtree containing 
$e$ is the union of a subtree of $T_1$ containing $v_1$ and a subtree of $T_2$ containing $v_2$, so
\begin{eqnarray*}
\alpha(T,e) &=& \alpha(T_1,v_1)\alpha(T_2,v_2), \\
\bar{\alpha}(T,e) &=& \bar{\alpha}(T_1,v_1)+\bar{\alpha}(T_2,v_2)+\alpha(T_1,v_1)+\alpha(T_2,v_2).
\end{eqnarray*}
Without loss of generality we may assume $\alpha(T_1,v_1) \ge \alpha(T_2,v_2)$, so, using Lemma \ref{stroot}, 
$\bar{\alpha}(T,e)\le 4\alpha(T_1,v_1)$ and so $2\alpha(T,e)\ge\frac{1}{2}\alpha(T_2,v_2)\bar{\alpha}(T,e)$. 
If $T_2$ has $k$ leaves then $\alpha(T_2, v_2)\ge 2^k$, so $e$ suffices if the number of leaves on each side 
is at least $\log_2n+1$. 

If no such $e$ exists then we must have a central vertex $v$ with each component of $T-v$ having fewer than 
$\log_2n+1$ leaves. (Since $T$ has at least $\frac{n+2}{2}$ leaves, and this is much more than $\log_2n+1$ 
for $n\ge 30$, certainly such a $v$ is an internal vertex.) Let $w$ be the neighbour maximising $\alpha(T_w,w)$
(where $T_w$ is the component of $T-v$ containing $w$); suppose $T_w$ has $k$ leaves and $T$ has $l$. Now
\begin{eqnarray*}
\alpha(T,v) &\ge& \alpha(T,vw) \\
&\ge& 2^{l-k}\alpha(T_w,w), \\
\bar{\alpha}(T,v) &=& \sum_{u\in\Gamma(v)}\left(\alpha(T_u,u)+\bar{\alpha}(T_u,u)\right) \\
&\le& \sum_{u\in\Gamma(v)}2\alpha(T_u,u) \\
&\le& 2n\alpha(T_w,w),
\end{eqnarray*}
and $l\ge\frac{n+2}{2}$, $k<\log_2n+1$, so 
\begin{eqnarray*}
\frac{\alpha(T,v)}{\bar{\alpha}(T,v)}&>&\frac{2^{\frac{n}{2}-\log_2n}}{2n} \\
&=&\frac{2^{\frac{n}{2}}}{2n^2}.
\end{eqnarray*}
So we are done if $2^\frac{n}{2}\ge n^3$, i.e. if $n\ge 6\log_2n$, which holds if $n\ge 30$.
\end{proof}
This result immediately gives us the desired approximation.
\begin{corollary}\label{stapprox2}
If $T\in\mathcal{T}_3$ has $n\ge 30$ vertices then there is some internal $v$ with $|\mu(T)-\lambda(T,v)|<2$.
\end{corollary}
\begin{proof}
Since 
\begin{eqnarray*}
\mu(T)&=&\frac{\alpha(T,v)\lambda(T,v)+\bar{\alpha}(T,v)\bar{\lambda}(T,v)}{\alpha(T,v)+\bar{\alpha}(T,v)} \\
&=&\lambda(T,v)-\frac{\bar{\alpha}(T,v)\left(\lambda(T,v)-\bar{\lambda}(T,v)\right)}{\alpha(T,v)+\bar{\alpha}(T,v)}
\end{eqnarray*}
and $0<\lambda(T,v),\bar{\lambda}(T,v)<n$, it is sufficient to find an internal $v$ such that 
$2\alpha(T,v)\ge n\bar{\alpha}(T,v)$. Either we are immediately done by Lemma \ref{stapprox1} or we have an 
edge satisfying the same relation; in the latter case at least one vertex on that edge is internal, and will 
suffice since if $e=vw$ then $\alpha(T,v)\ge\alpha(T,e)$ and $\bar{\alpha}(T,e)\ge\bar{\alpha}(T,v)$. 
\end{proof}

\section{Ranking vertices and lower bounds}
In order to show that we require the proportion of leaves to tend to 1 if the average subtree density is to 
tend to $\frac{1}{2}$ we need a better lower bound for trees which have many internal vertices. The following 
bound may be deduced from our subsequent stronger result, but we shall prove this weaker version first to give 
the basic idea.
\begin{lemma}
If $T\in \mathcal{T}^*_3$ has root $v$ and $n$ vertices of which $k$ are not leaves then 
$\lambda(T,v)\ge \frac{n+1}{2}+\frac{k-1}{10}$.
\end{lemma}
\begin{proof}
Again, we prove this by induction on $n$; it is trivial for $n=1$. If $n>1$, let $v$ have $l$ neighbours 
which are leaves, and $m$ which are not. Let $v_1,\ldots,v_m$ be these neighbours and 
$v_{m+1},\ldots,v_{m+l}$ be the leaves. Let $T_i$ be the component of $T-v$ containing $v_i$, and write 
$n_i, k_i$ for the number of vertices and the number of vertices which are not leaves respectively.
Now $\sum_{i=1}^m n_i=n-l-1$, $\sum_{i=1}^m k_i=k-1$, and
\begin{eqnarray*}
\lambda(T,v) &=& 1+\sum_{i=1}^{m+l}\frac{\lambda(T_i,v_i)\alpha(T_i,v_i)}{\alpha(T_i,v_i)+1} \\
&=& 1+\frac{l}{2}+\sum_{i=1}^m\frac{\lambda(T_i,v_i)\alpha(T_i,v_i)}{\alpha(T_i,v_i)+1} \\
&\ge& 1+\frac{l}{2}+\sum_{i=1}^m\frac{\left(n_i+1+\frac{k_i-1}{5}\right)\alpha(T_i,v_i)}
{2\left(\alpha(T_i,v_i)+1\right)},
\end{eqnarray*}
by the induction hypothesis. Note that, for $1\le i\le m$, since $T_i$ is not a single vertex but 
$T_i \in \mathcal{T}^*_3$, $n_i\ge 3$. Also any tree in $\mathcal{T}^*_3$ with $n_i$ vertices has at least 
$\frac{n_i+1}{2}$ leaves. Since the set of all non-leaf vertices together with any subset of leaves forms a 
subtree of $T_i$ containing $v_i$, $\alpha(T_i,v_i)\ge 2^{\frac{n_i+1}{2}} \ge \frac{11n_i-1}{8}$, where the 
second inequality is easily checked to hold for $n_i\ge 3$. Finally, since $T_i$ has at least 
$\frac{n_i+1}{2}$ leaves, $k_i\le \frac{n_i-1}{2}$ and so $\frac{11n_i-1}{8}\ge \frac{5n_i+k_i}{4}$. Thus
\begin{eqnarray*}
\lambda(T,v) &\ge& 1+\frac{l}{2}+\sum_{i=1}^m\frac{\left(n_i+1+\frac{k_i-1}{5}\right)\alpha(T_i,v_i)}
{2\left(\alpha(T_i,v_i)+1\right)} \\
&\ge& 1+\frac{l}{2}+\sum_{i=1}^m\frac{\left(n_i+1+\frac{k_i-1}{5}\right)\left(\frac{5n_i+k_i}{4}\right)}
{2\left(\frac{5n_i+k_i}{4}+1\right)} \\
&=& 1+\frac{l}{2}+\sum_{i=1}^m\frac{n_i+\frac{k_i}{5}}{2} \\
&=& \frac{n+1}{2}+\frac{k-1}{10}
\end{eqnarray*}
as required.
\end{proof}
This result is best possible for $n\ge 3k+1$, attained by the tree whose root has $n-3k$ children which are 
leaves and $k$ other children, each having two children. If $T\in \mathcal{T}_3$ then we have shown that 
$\lambda(T,v)\ge \frac{n+1}{2}+\frac{k-1}{10}$ for any internal vertex $v$, since if $T$ is rooted at $v$ 
it is in $\mathcal{T}^*_3$.

We may immediately conclude our desired result.
\begin{theorem}\label{st12}
A sequence of distinct trees in $\mathcal{T}_3$ has average subtree density tending to $\frac{1}{2}$ if and only if the proportion of leaves tends to 1.
\end{theorem}
\begin{proof}
Write $(T_i)_{i\ge 0}$ for our sequence and let $T_i$ have $n_i$ vertices and $n_i\xi_i$ leaves. Since 
there are only finitely many trees of each order, $n_i\to \infty$ and for $i$ sufficiently large there 
exists an internal vertex $v_i$ such that $|\mu(T_i)-\lambda(T_i,v_i)|<2$. Then
$|D(T_i)-\frac{1}{n_i}\lambda(T_i,v_i)|\to 0$, and 
\begin{equation*}
\frac{3}{5}-\frac{1}{10}\xi_i<\frac{\lambda(T_i,v_i)}{n_i}<1-\frac{1}{2}\xi_i,
\end{equation*}
so $D(T_i)\to\frac{1}{2}$ if and only if $\xi_i\to 1$.
\end{proof}

For density tending to $\frac{3}{4}$, we need to divide the vertices into more classes. We shall consider 
only rooted trees for this purpose.

For a rooted tree $T\in \mathcal{T}^*_3$ with root $v$, we regard each vertex $w\ne v$ as having one 
\textit{parent}, the neighbour of $w$ on the path from $v$ to $w$, and $d(w)-1$ \textit{children}, the 
other neighbours of $w$. Thus each child of a vertex is a neighbour which is further from the root. We 
shall inductively define the \textit{rank} of a vertex other than the root: a vertex with no 
children (i.e. a leaf) has rank zero; the rank of each other vertex is one more than the maximum rank of 
its children. An equivalent explicit definition is that the rank is the maximum length of a path starting 
at that vertex which does not include its parent. We shall leave the rank of the root undefined. 

We wish to find a lower bound on $\lambda(T,v)$ in terms of the ranks of the vertices. Write $m_j(T,v)$ for 
the number of vertices of rank $j$ when $T$ is rooted at $v$ (for any $T$, $m_j(T,v)=0$ for all sufficiently 
large $j$, certainly for any $j$ exceeding the diameter of $T$). Note that since each vertex of rank $j+1$ 
has at least one child of rank $j$, and these are all distinct, $m_{j+1}(T,v)\le m_j(T,v)$.

\begin{theorem}
If $T\in \mathcal{T}^*_3$ has root $v$ and $n$ vertices then 
\begin{equation*}
\lambda(T,v)\ge 1+\sum_{j\ge 0}c_jm_j(T,v),
\end{equation*}
where the values of $c_j$ for $j\ge 0$ are given by 
\begin{equation*}
c_j=1-\frac{1+\frac{j}{2}+\sum_{i=0}^{j-1}c_i}{2^{j+1}+j};
\end{equation*}
note that when $j=0$ the sum in the above expression is empty and so $c_0=\frac{1}{2}$.
\end{theorem}
\begin{proof}
First, we need some bounds on the $c_j$. We claim by induction that
\begin{equation*}
c_j\le 1-\frac{1+j}{2^{j+1}+j}\le 1
\end{equation*}
and that
\begin{equation*}
c_j\ge 1-\frac{1+\frac{3j}{2}}{2^{j+1}+j}\ge \frac{1}{2}.
\end{equation*}
We shall only use $\frac{1}{2}\le c_j \le 1$ (which is certainly true for j=0) in our induction step: 
if this is true for every $0\le j<k$ then
\begin{equation*}
c_k\le 1-\frac{1+k}{2^{k+1}+k}\le 1
\end{equation*}
and
\begin{equation*}
c_k\ge 1-\frac{1+\frac{3k}{2}}{2^{k+1}+k},
\end{equation*}
as required. It only remains to check that 
\begin{equation*}
1-\frac{1+\frac{3k}{2}}{2^{k+1}+k}\ge\frac{1}{2}
\end{equation*}
for every $k\ge 0$; this is true since $2^k\ge k+1$ so $2^{k+1}+k\ge 3k+2$. Thus, by induction, the bounds above hold for every $k\ge 0$.

Now we claim that these bounds imply that $c_j\to 1$ as $j \to \infty$ and that $c_{j+1}\ge c_j$ for every $j$. The first statement is a trivial
consequence of the lower bound. To prove the other, note that
\begin{eqnarray*}
\frac{1-c_j}{1-c_{j+1}} &\le&
 \left(\frac{j+1}{2^{j+1}+j}\right)\left(\frac{2^{j+2}+j+1}{\frac{3}{2}(j+1)+1}\right) \\
 &=& \frac{4j 2^j+4\times 2^j +j^2+2j+1}{3j 2^j+5\times 2^j+\frac{3}{2}j^2+\frac{5}{2}j}.
\end{eqnarray*}
Since 
\begin{eqnarray*} 
&&\left(4j 2^j+4\times 2^j +j^2+2j+1\right) - \left(3j 2^j+5\times 2^j+\tfrac{3}{2}j^2+\tfrac{5}{2}j\right) \\
&&\;\;\;\;\;\;\;\;\;\;\;\;\;\;\;\;\;\;\;\;\;\;\;=\;2^j(j-1)+1-\frac{1}{2}j^2-\frac{1}{2}j \\
&&\;\;\;\;\;\;\;\;\;\;\;\;\;\;\;\;\;\;\;\;\;\;\;\ge\; (j+1)(j-1)+1-j^2=0,
\end{eqnarray*}
it follows that $1-c_j\ge 1-c_{j+1}$, i.e. $c_j\le c_{j+1}$.

We are now ready to prove the main result, which we shall do by induction on $n$; it is trivial for $n=1$ 
(when $m_j(T,v)=0$ for every $j$). If $n>1$, let $v$ have $d$ neighbours $v_1,\ldots,v_d$ of ranks 
$r_1,\ldots,r_d$; let $T_i$ be the component of $T-v$ containing $v_i$, and write $n_i$ and $m_j(T_i,v_i)$ 
for the number of vertices and the number of vertices of rank $j$ respectively. Note that for each vertex 
in $T_i$ other than the root, the rank of that vertex in $T_i$ is equal to its rank in $T$. Therefore, the 
difference $m_j(T,v)-\sum_{i=1}^{d}m_j(T_i,v_i)$ is the number of the $v_i$ which have rank $j$ in $T$, and so 
\begin{equation*}
\sum_{j\ge 0}c_jm_j(T,v)=\sum_{i=1}^d\left(c_{r_i}+\sum_{j\ge 0}c_jm_j(T_i,v_i)\right).
\end{equation*}
Also,
\begin{eqnarray*}
\lambda(T,v) &=& 1+\sum_{i=1}^d\frac{\lambda(T_i,v_i)\alpha(T_i,v_i)}{\alpha(T_i,v_i)+1} \\
&\ge& 1+\sum_{i=1}^d\frac{\left(1+\sum_{j\ge 0}c_jm_j(T_i,v_i)\right)\alpha(T_i,v_i)}{\alpha(T_i,v_i)+1},
\end{eqnarray*}
by the induction hypothesis, and so it is sufficient to prove that for every $i$
\begin{equation*}
\frac{\left(1+\sum_{j\ge 0}c_jm_j(T_i,v_i)\right)\alpha(T_i,v_i)}{\alpha(T_i,v_i)+1} \ge 
\sum_{j\ge 0}c_jm_j(T_i,v_i)+c_{r_i},
\end{equation*}
or equivalently
\begin{equation*}
1-c_{r_i} \ge \frac{1+\sum_{j\ge 0}c_jm_j(T_i,v_i)}{\alpha(T_i,v_i)+1}.
\end{equation*}
If $r_i=0$ then the root has no children and $m_j(T_i,v_i)=0$ for each $j$; also $\alpha(T_i,v_i)=1$ so this 
is true (and the two sides are equal). If $r_i=1$ then all the children of the root are leaves; if it has 
$l$ children then $l\ge 2$ and so
\begin{eqnarray*}
\frac{1+\sum_{j\ge 0}c_jm_j(T_i,v_i)}{\alpha(T_i,v_i)+1}&=&\frac{1+\frac{l}{2}}{1+2^l} \\
&\le&\frac{2}{5}=1-c_1,
\end{eqnarray*}
as required.

From here onwards, then, we shall assume $r_i>1$. In that case the root is not a leaf and since 
$T_i \in \mathcal{T}^*_3$ it has at least $\frac{n_i+1}{2}$ leaves; it also contains the root and at least 
one vertex of rank $j$ for every $1\le j< r_i$. By Lemma \ref{sttreecount}, then, 
\begin{equation*}
\alpha(T_i,v_i)\ge 2^{\frac{n_i+1}{2}}+r_i-1. 
\end{equation*}

We now turn to bound $\sum_{j\ge 0}c_jm_j(T_i,v_i)$. Recall that, for each $j$, $c_{j+1}\ge c_j$ but 
$m_{j+1}(T_i,v_i)\le m_j(T_i,v_i)$. Also note that, since $v_i$ has rank $r_i$ in $T$, $m_j(T_i,v_i)\ge 1$ for 
each $j<r_i$ but $m_j(T_i,v_i)=0$ for each $j\ge r_i$. Now 
\begin{equation*}
\left(r_i-1\right)\sum_{j=1}^{r_i-1}c_jm_j(T_i,v_i)\le\left(\sum_{j=1}^{r_i-1}c_j\right)\left(\sum_{j=1}^{r_i-1}m_j(T_i,v_i)\right) 
\end{equation*}
by Chebyshov's sum inequality (see, e.g., \cite{HLP}). Using the fact that
\begin{equation*}
\sum_{j=0}^{r_i-1}m_j(T_i,v_i)=n_i-1,
\end{equation*}
and writing $l_i=m_0(T_i,v_i)$, we get
\begin{equation*}
\sum_{j\ge 0}c_jm_j(T_i,v_i)\le c_0l_i+\left(\frac{c_1+c_2+\cdots+c_{r_i-1}}{r_i-1}\right)\left(n_i-l_i-1\right).
\end{equation*}
Since $c_0=\frac{1}{2}\le c_i$, the right-hand side is a decreasing function of $l_i$. Remembering that $l_i\ge\frac{n_i+1}{2}$, then,
\begin{equation*}
\sum_{j\ge 0}c_jm_j(T_i,v_i)\le \frac{1}{2}\left(\frac{n_i+1}{2}\right)+\left(\frac{c_1+c_2+\cdots+c_{r_i-1}}{r_i-1}\right)\left(\frac{n_i-3}{2}\right),
\end{equation*}
and, combining our two bounds,
\begin{equation*}
\frac{1+\sum_{j\ge 0}c_jm_j(T_i,v_i)}{\alpha(T_i,v_i)+1} \le 
\frac{1+\frac{1}{2}\left(\frac{n_i+1}{2}\right)+\left(\frac{c_1+c_2+\cdots+c_{r_i-1}}{r_i-1}\right)\left(\frac{n_i-3}{2}\right)}
{2^{\frac{n_i+1}{2}}+r_i}.
\end{equation*}
Recall that $T_i$ has at most $\frac{n_i-1}{2}$ vertices which are not leaves, which include the root and 
at least one vertex of rank $j$ for each $1\le j<r_i$. Thus $n_i\ge 2r_i+1$.

Fix $r\ge 2$ and write
\begin{eqnarray*}
f_r(k)&=&1+\frac{1}{2}\left(\frac{k+1}{2}\right)+\left(\frac{c_1+c_2+\cdots+c_r}{r-1}\right)\left(\frac{k-3}{2}\right); \\
g_r(k)&=&2^{\frac{k+1}{2}}+r.
\end{eqnarray*}
If $k\ge 2r+1 \ge 5$, since $1\ge c_j\ge\frac{1}{2}$ for each $j$, $f_r(k)\ge\frac{k}{2}$ and $f_r(k+1)-f_r(k)\le\frac{3}{4}$, so
\begin{equation*}
\frac{f_r(k+1)-f_r(k)}{f_r(k)}\le\frac{3}{2k}\le\frac{3}{10}; 
\end{equation*}
also, $2^{\frac{k+1}{2}}\ge 2^{r+1}\ge 4r$ and so
\begin{eqnarray*}
\frac{g_r(k+1)-g_r(k)}{g_r(k)}&=&\frac{(\sqrt{2}-1)2^{\frac{k+1}{2}}}{2^{\frac{k+1}{2}}+r} \\
&\ge&\frac{(\sqrt{2}-1)}{1+\frac{1}{4}}>\frac{8}{25}.
\end{eqnarray*}
Consequently, $\frac{f_r(k+1)}{f_r(k)}<\frac{g_r(k+1)}{g_r(k)}$ and so $\frac{f_r(k)}{g_r(k)}$ is a 
decreasing function of $k$ and maximised at $k=2r+1$ (given the requirement that $k\ge 2r+1$). Therefore,
\begin{eqnarray*}
\frac{1+\sum_{j\ge 0}c_jm_j(T_i,v_i)}{\alpha(T_i,v_i)+1} &\le& 
\frac{1+\frac{1}{2}\left(\frac{n_i+1}{2}\right)+\left(\frac{c_1+c_2+\cdots+c_{r_i-1}}{r_i-1}\right)\left(\frac{n_i-3}{2}\right)}
{2^{\frac{n_i+1}{2}}+r_i} \\
&\le& \frac{1+\frac{1}{2}(r_i+1)+\left(\frac{c_1+c_2+\cdots+c_{r_i-1}}{r_i-1}\right)(r_i-1)}{2^{r_i+1}+r_i} \\
&=& \frac{1+\frac{r_i}{2}+c_0+c_1+c_2+\cdots+c_{r_i-1}}{2^{r_i+1}+r_i} =1-c_{r_i},
\end{eqnarray*}
as required.
\end{proof}
The sequence $(c_j)_{j\ge 0}$ begins 
$\frac{1}{2},\frac{3}{5},\frac{69}{100},\frac{1471}{1900},\frac{4819}{5700},\frac{70783}{78660},\ldots$, 
but we shall only need that it is increasing and tends to one. We may combine the above bound with that 
of Lemma \ref{twigub} to obtain necessary and sufficient conditions for a sequence of trees in 
$\mathcal{T}_3$ to have density tending to $\frac{3}{4}$.
\begin{theorem}\label{st34}
A sequence of distinct trees in $\mathcal{T}_3$ has average subtree density tending to $\frac{3}{4}$ if 
and only if the proportion of leaves tends to $\frac{1}{2}$ and the proportion of twigs tends to 0.
\end{theorem}
\begin{proof}
Write $(T_i)_{i\ge 0}$ for our sequence and let $T_i$ have $n_i$ vertices, $n_i\xi_i$ leaves and 
$n_i\varepsilon_i$ twigs. We know that $\xi_i>\frac{1}{2}$ for every $i$, since the number of leaves is 
at least $\frac{n_i+2}{2}$. By Lemmas \ref{stleaf} and \ref{twigub} we have the two upper bounds 
$\mu(T_i)<n_i-\frac{1}{2}n_i\xi_i$ and $\mu(T_i)<\frac{3}{4}n_i-\frac{2}{5}n_i\varepsilon_i$, so
$D(T_i)<1-\frac{1}{2}\xi_i$ and $D(T_i)<\frac{3}{4}-\frac{2}{5}\varepsilon_i$. If $\xi_i\not\to\frac{1}{2}$ 
then for some $\delta>0$, $\xi_i>\frac{1}{2}+\delta$, and so $D(T_i)<\frac{3}{4}-\frac{1}{2}\delta$, for 
infinitely many $i$; if $\varepsilon_i\not\to 0$ then similarly there is some $\delta>0$ for which 
$D(T_i)<\frac{3}{4}-\frac{2}{5}\delta$ for infinitely many $i$.

Conversely, suppose $\xi_i\to\frac{1}{2}$ and $\varepsilon_i\to 0$. Since there are only finitely many trees 
of each order, $n_i\to \infty$ and for $i$ sufficiently large there exists $v_i$ such that 
$|\mu(T_i)-\lambda(T_i,v_i)|<2$; again $|D(T_i)-\frac{1}{n_i}\lambda(T_i,v_i)|\to 0$. 

Consider $T_i$ rooted at $v_i$. There are at most $n_i\xi_i$ vertices of rank zero and $n_i\varepsilon_i$ 
of rank one. In fact, since each vertex of rank $j+1$ has a child of rank $j$, and these are distinct, 
there are at most $n_i\varepsilon_i$ vertices of rank $j$ for every $j\ge 1$. For any $\delta>0$ we may 
find $k$ such that $c_{k+1}>1-\delta$. Then, using the same notation as before,
\begin{eqnarray*}
\lambda(T_i,v_i) &\ge& 1+\sum_{j\ge 0}c_jm_j(T_i,v_i) \\
&\ge& \frac{1}{2}\left(\sum_{j=0}^{k}m_j(T_i,v_i)\right)+c_{k+1}\left(1+\sum_{j\ge k+1}m_j(T_i,v_i)\right) \\
&=& \frac{1}{2}\left(\sum_{j=0}^{k}m_j(T_i,v_i)\right)+c_{k+1}\left(n_i-\sum_{j=0}^{k}m_j(T_i,v_i)\right) \\
&>& \frac{1}{2}n_i(\xi_i+k\varepsilon_i)+(1-\delta)n_i(1-\xi_i-k\varepsilon_i)
\end{eqnarray*}
For $i$ sufficiently large, $\frac{1}{2}<\xi_i+k\varepsilon_i<\frac{1}{2}(1+\delta)$, so 
\begin{equation*}
\frac{\lambda(T_i,v_i)}{n_i}>\frac{1}{4}+\frac{1}{2}(1-\delta)^2>\frac{3}{4}-\delta.
\end{equation*}
Since $\frac{1}{n_i}\lambda(T_i,v_i)<\frac{3}{4}+\frac{2}{n_i}$, $\frac{1}{n_i}\lambda(T_i,v_i)\to\frac{3}{4}$, and so $D(T_i)\to\frac{3}{4}$.
\end{proof}

\section{Final Remarks}
Theorems \ref{st12} and \ref{st34} give a complete classification of sequences of series-reduced trees with
average subtree density tending to either extremal value. We remarked earlier that it is not necessary for 
the ratio of the diameter to the number of vertices to tend to $\frac{1}{2}$ in order for the average 
subtree density to tend to $\frac{3}{4}$; we conclude by giving a sequence with average subtree density 
tending to $\frac{3}{4}$ for which the ratio of the diameter to the number of vertices tends to 0.

For $k\ge 3$ and $r\ge 1$, we define the \textit{starfish} $\mathsf{Sf}_{k,r}$ with $k$ arms and radius $r$ 
as follows: take $k$ paths of length $r$ (i.e. having $r+1$ vertices) with a shared end-vertex but otherwise 
disjoint. Now to each vertex of degree two attach an additional leaf. 

$\mathsf{Sf}_{k,r}$ has $2k(r-1)+1$ vertices, $kr$ leaves and $k$ twigs; its diameter is $2r$. It follows
from Theorem \ref{st34} that a sequence $\left(\mathsf{Sf}_{k_i,r_i}\right)_{i\ge 0}$ of starfish has 
average subtree density tending to $\frac{3}{4}$ provided $r_i\to \infty$; if also $k_i\to \infty$ then the
ratio of the diameter to the number of vertices tends to 0.

\end{document}